\documentclass[11pt]{amsart}

\usepackage{amsthm,amssymb,amstext,amscd,amsfonts,amsbsy,amsxtra,latexsym,amsmath}
\usepackage{fullpage}
\usepackage[english]{babel}
\usepackage[latin1]{inputenc}
\usepackage{verbatim}
\usepackage{mathrsfs}
\usepackage[all]{xy}
\newcommand{\field}[1]{\mathbb{#1}}
\newcommand{\N}{\field{N}}
\newcommand{\Z}{\field{Z}}
\newcommand{\R}{\field{R}}
\newcommand{\C}{\field{C}}

\newcommand{\PP}{P}
\newcommand{\bigT}{T}

\newcommand{\sgn}{\operatorname{sgn}}
\newcommand{\pr}{\operatorname{pr}}
\newcommand{\SL}{\operatorname{SL}}

\newcommand{\im}{\text{Im}}
\newcommand{\re}{\text{Re}}
\renewcommand{\H}{\mathbb{H}}
\newcommand{\QD}{\mathcal{Q}_D}

\newcommand{\fkD}{f_{k,D}}

\newcommand{\wt}{\kappa}

\newcommand{\tauz}{z}
\newcommand{\ztau}{\tau}
\newcommand{\Phipos}{\Phi_k}
\newcommand{\Phiposno}{\Phi}
\newcommand{\Phineg}{\Phi_{1-k}^*}
\newcommand{\Phinegno}{\Phi^*}
\newcommand{\phipos}{\varphi}
\newcommand{\phineg}{\varphi^*}
\newcommand{\Thpos}{\Theta}
\newcommand{\Thneg}{\Theta^*}
\numberwithin{equation}{section}
\newtheorem{theorem}{\textbf{Theorem}}
\numberwithin{theorem}{section}
\newtheorem{lemma}[theorem]{\textbf{Lemma}}
\newtheorem{proposition}[theorem]{\textbf{Proposition}}

\theoremstyle{remark}
\newtheorem*{remark}{Remark}
\newtheorem*{remarks}{Remarks}

\renewenvironment{proof}[1][Proof]{\begin{trivlist}
\item[\hskip \labelsep {\bfseries #1:}]}{\qed\end{trivlist}}

\begin{document}
\title{Theta lifts and local Maass forms}
\author{Kathrin Bringmann}
\address{Mathematical Institute\\University of Cologne\\ Weyertal 86-90 \\ 50931 Cologne \\Germany}
\email{kbringma@math.uni-koeln.de}
\author{Ben Kane}
\address{Mathematical Institute\\University of Cologne\\ Weyertal 86-90 \\ 50931 Cologne \\Germany}
\email{bkane@math.uni-koeln.de}
\author{Maryna Viazovska}
\address{Max Planck Institute for Mathematics\\ Vivatsgasse 7 \\ 53111 Bonn \\Germany}
\email{viazovsk@mpim-bonn.mpg.de}
\thanks{The research of the first author was supported by the Alfried Krupp Prize for Young University Teachers of the Krupp Foundation.}
\date{\today}
\begin{abstract}
The first two authors and Kohnen have recently introduced a new class of modular objects called locally harmonic Maass forms, which are annihilated almost everywhere by the hyperbolic Laplacian operator.  In this paper, we realize these locally harmonic Maass forms as theta lifts of harmonic weak Maass forms.  Using the theory of theta lifts, we then construct examples of (non-harmonic) local Maass forms, which are instead eigenfunctions of the hyperbolic Laplacian almost everywhere.
\end{abstract}

\maketitle

\section{Introduction and statement of results}\label{sec:intro}

In \cite{BKW}, a new class of modular objects was introduced.  These functions, known as locally harmonic Maass forms, satisfy negative weight modularity and are annihilated almost everywhere by the hyperbolic Laplacian (see Section \ref{sec:harmonic} for the relevant definitions), mirroring harmonic weak Maass forms.  Recent interest in harmonic weak Maass forms initiated with their systematic treatment by Bruinier and Funke \cite{BF}.  Following their appearance in the theory of mock theta functions due to Zwegers \cite{ZwegersThesis}, it has been shown that harmonic weak Maass forms have applications ranging from partition theory (for example \cite{AndrewsMock,Bringmann,BGM,BOfq,BORank}) and Zagier's duality \cite{ZaDual} relating ``modular objects'' of different weights (for example \cite{BODual}) to derivatives of $L$-functions (for example \cite{BrO,BrY}).  They  also arise in mathematical physics, as recently evidenced in Eguchi, Ooguri, and Tachikawa's \cite{EOT} investigation of moonshine for the largest Mathieu group $M_{24}$.   The main difference between locally harmonic Maass forms and harmonic weak Maass forms is that there are certain geodesics along which locally harmonic weak Maass forms are not necessarily real analytic and may even exhibit discontinuities.  

In this paper, we realize the locally harmonic Maass forms studied in \cite{BKW} as theta lifts of harmonic weak Maass forms.  Theta lifts form connections between different types of modular objects and the regularization of Harvey--Moore \cite{HM} and Borcherds \cite{Bo98} allow one to extend their definitions to previously divergent theta integrals.  In particular, the Shimura lift \cite{Shimura} was realized as a theta lift by Niwa \cite{Niwa}.  Borcherds \cite{Bo98} later placed this into the framework of a larger family of theta lifts.  Following his work, theta lifts have more recently appeared in a variety of applications including generalized Kac--Moody algebras \cite{GritsenkoNikulin} and the arithmetic of Shimura varieties \cite{BrY}.  To expound upon one example, Katok and Sarnak \cite{KatokSarnak} used theta lifts to relate the central value of the $L$-series of a Maass cusp form to the Fourier coefficients of corresponding Maass cusp forms under the Shimura lift.  This extended a famous result of Waldspurger \cite{Waldspurger} proving that the central value of the $L$-function of an integral weight Hecke eigenform is proportional to the square of a coefficient of its half-integral weight counterpart under the Shintani lift.  Tunnell \cite{Tunnell} later exploited this link to express the central value of the $L$-function of an elliptic curve in terms of the coefficients of a theta function associated to a ternary quadratic form.  Tunnell's Theorem gives a solution to the ancient congruent number problem (conditional on the Birch and Swinnerton-Dyer conjecture).  

Following Bruinier's \cite{Bruinier} application of Borcherds lifts to harmonic weak Maass forms, Bruinier and Funke \cite{BF} extended theta lifts to harmonic weak Maass forms.    Due to the theory built around theta lifts, one may naturally extend the definition of locally harmonic Maass forms to include local Maass forms, i.e., functions with the above properties of locally harmonic Maass forms except that instead of being annihilated by the hyperbolic Laplacian, they are eigenfunctions.  The locally harmonic Maass forms investigated in \cite{BKW} have a natural connection to the Shimura \cite{Shimura} and Shintani \cite{Shintani} lifts, which we next describe.

For $k\in 2\N$ and a discriminant $D>0$, Zagier \cite{ZagierDN} defined the functions
\begin{equation}\label{eqn:fkDdef}
\fkD(\tauz):=\frac{D^{k-\frac{1}{2}}}{\binom{2k-2}{k-1}\pi}
\sum_{Q\in \QD} Q(\tauz,1)^{-k},
\end{equation}
where $\QD$ denotes the set of binary quadratic forms of discriminant $D\in \Z$.  Zagier showed that $\fkD\in S_{2k}$, the space of weight $2k$ cusp forms for $\SL_2\left(\Z\right)$ and it was later noticed that the $\fkD$ could be naturally realized as (linear combinations of) hyperbolic Poincar\'e series defined by Petersson \cite{Petersson}.  The functions $\fkD$ reappeared in the (holomorphic) kernel function for the Shimura and Shintani lifts
$$
\Omega(\tauz,\ztau):=\sum_{0<D\equiv 0,1\pmod{4}}  \fkD(\tauz) e^{2\pi i D\ztau}
$$
between $S_{2k}$ and $S_{k+\frac{1}{2}}^+$ (Kohnen's plus space of weight $k+\frac{1}{2}$ modular forms), which was defined by Kohnen and Zagier \cite{KohnenZagier}.  For $g\in S_{k+\frac{1}{2}}^+$, the Petersson inner product $\left< g,\Omega(-\overline{\tauz},\cdot)\right>$ equals $(-1)^{k/2}2^{2-3k}$ times the Shimura lift of $g$.  Kohnen and Zagier used $\Omega$ to explicitly compute the constant of proportionality in Waldspurger's result, in turn proving nonnegativity of the central $L$-values of Hecke eigenforms.

As indicated above, the functions $f_{k,D}$ may be interpreted in terms of theta lifts.  To describe this, we define Shintani's \cite{Shintani} non-holomorphic kernel function.  Throughout we write $\ztau=u+iv\in \H$, $\tauz=x+iy\in\H$, and denote for $Q=[a,b,c]\in \QD$
$$
Q_{\tauz}:=\frac{1}{y}\left(a|\tauz|^2+bx+c\right).
$$
Using this notation, Shintani's theta function projected into Kohnen's plus space equals
\begin{equation}\label{eqn:thetadef1}
\Thpos (\tauz,\ztau):=y^{-2k}v^{\frac12}\sum_{\substack{D\in \Z\\ Q\in\QD}}Q(\tauz,1)^k e^{-4\pi Q_{\tauz}^2 v}e^{2\pi i D\ztau}.
\end{equation}

The function $\Thpos\left(-\overline{\tauz},\ztau\right)$ transforms like a modular form of weight $k+\frac{1}{2}$ in $\ztau$ and weight $2k$ in $\tauz$ (see Proposition \ref{prop:Thetamodular} (1)).  Integrating the $D$-th weight $k+\frac{1}{2}$ (holomorphic) Poincar\'e series against $\Thpos$ yields $f_{k,D}$.  One can use Borcherds's \cite{Bo98} aforementioned regularized version $\left<f,g\right>^{\text{reg}}$ of the Petersson inner product (see Section \ref{sec:harmonic} for a definition) to extend the utility of the Shimura lift (realized as Niwa's \cite{Niwa} theta lift) to weak Maass forms.  To be more precise, for a weight $k+\frac{1}{2}$ weak Maass form $H$ with eigenvalue 
$$
\lambda_s:=\left(s-\frac{k}{2}-\frac{1}{4}\right)\left(1-s-\frac{k}{2}-\frac{1}{4}\right)
$$
 under the hyperbolic Laplacian $\Delta_{k+\frac{1}{2}}$, we define the theta lift
\[
\Phipos(H)(\tauz):= \left<H,\Thpos\left(\tauz,\cdot\right)\right>^{\text{reg}}.
\]
By choosing an appropriate input, this lift leads to the natural generalization
\begin{equation}\label{eqn:fkDsdef}
f_{k,s,D}(\tauz):=\sum_{Q\in \QD} Q(\tauz,1)^{-k}\phipos_s\left(\frac{Dy^2}{\left|Q\left(\tauz,1\right)\right|^2}\right)
\end{equation}
of $f_{k,D}$.  Here, for $0<w\leq 1$ and $\re(s)\geq \frac{k}{2}+\frac{1}{4}$, using the usual $_2F_1$ notation for Gauss's hypergeometric function, we define
$$
\phipos_s(w):=\frac{\Gamma\left(s+\frac{k}{2}-\frac{1}{4}\right) D^{\frac{k}{2}+\frac{1}{4}}}{6\Gamma(2s)\left(4\pi\right)^{\frac{k}{2}-\frac{1}{4}}} w^{s-\frac{k}{2}-\frac{1}{4}} {_2F_1}\left(s+\frac{k}{2}-\frac{1}{4},s-\frac{k}{2}-\frac{1}{4}; 2s;w\right),
$$
which is easily seen to be a constant when $s=\frac{k}{2}+\frac{1}{4}$.  Note that for $\re(s)>\frac{k}{2}+\frac{1}{4}$, the Euler integral representation of the ${_2F_1}$ (see \eqref{eqn:Eulerint}) yields
$$
\phipos_s(w)=\frac{\Gamma\left(s+\frac{k}{2}-\frac{1}{4}\right) D^{\frac{k}{2}+\frac{1}{4}} w^{s-\frac{k}{2}-\frac{1}{4}}}{6\Gamma\left(s+\frac{k}{2}+\frac{1}{4}\right)\Gamma\left(s-\frac{k}{2}-\frac{1}{4}\right)\left(4\pi\right)^{\frac{k}{2}-\frac{1}{4}}}\int_{0}^{1}(1-t)^{s+\frac{k}{2}-\frac{3}{4}}t^{s-\frac{k}{2}-\frac{5}{4}}\left(1-wt\right)^{-s-\frac{k}{2}+\frac{1}{4}}dt.
$$
In order to obtain the functions $f_{k,s,D}$, we apply the theta lift $\Phipos$ to the $D$-th Poincar\'e series $\PP_{k+\frac{1}{2},s,D}$ (see \eqref{eqn:Psdef}) of weight $k+\frac{1}{2}$ with eigenvalue $\lambda_s$ under $\Delta_{k+\frac{1}{2}}$ in Kohnen's plus space.  In the special case that $s=\frac{k}{2}+\frac{1}{4}$, this Poincar\'e series is precisely the classical cuspidal Poincar\'e series and $f_{k,\frac{k}{2}+\frac{1}{4},D}$ is essentially $f_{k,D}$ because $\phipos_{\frac{k}{2}+\frac{1}{4}}$ is a constant.  We next show that in general the functions $f_{k,s,D}$ are local Maass forms with exceptional set given by the closed geodesics
\begin{equation}\label{eqn:EDdef}
E_D:=\left\{ \tauz=x+iy\in \H: \exists a,b,c\in \Z,\ b^2-4ac=D,\ a\left|\tauz\right|^2+bx+c=0\right\}.
\end{equation}
\begin{theorem}\label{thm:poswt}
Suppose that $s\in \C$ satisfies $\re(s)\geq \frac{k}{2}+\frac{1}{4}$ and $D>0$ is a discriminant.  Then the following hold.

\noindent
\begin{enumerate}
\item  The function $f_{k,s,D}$ is a local Maass form of weight $2k$ and eigenvalue $4\lambda_s$ under $\Delta_{2k}$ with exceptional set $E_D$.  Moreover, 
\begin{equation}\label{eqn:fkDcusp}
f_{k,\frac{k}{2}+\frac{1}{4},D}=\frac{2^{2k-3}}{3(2k-1)}\left(4\pi D\right)^{\frac{3}{4}-\frac{k}{2}} f_{k,D},
\end{equation}
which is a cusp form.
\item
The theta lift $\Phipos$ maps weight $k+\frac{1}{2}$ weak Maass forms with eigenvalue $\lambda_s$ under $\Delta_{k+\frac{1}{2}}$ to weight $2k$ local Maass forms with eigenvalue $4\lambda_s$ under $\Delta_{2k}$.  In particular, the image of the $D$-th Poincar\'e series under the theta lift $\Phipos$ equals
$$
\Phipos\left(\PP_{k+\frac{1}{2},s,D}\right) = f_{k,s,D}.
$$
\end{enumerate}
\end{theorem}
\begin{remark}
The function $f_{k,s,D}$ is continuous for every $\re(s)\geq \frac{k}{2}+\frac{1}{4}$, but whenever $\lambda_s\neq 0$ there exist points along $E_D$ along which $f_{k,s,D}$ is not differentiable.  In particular, one should note the astonishing fact that while the functions are not differentiable for $\lambda_s\neq 0$, the case $\lambda_s=0$ yields a (holomorphic) cusp form by \eqref{eqn:fkDcusp}. 
\end{remark}
We now investigate the general properties of the theta lift.  Let $T_p$ and $T_p^2$ denote the Hecke operators of integral and half-integral weight, respectively (see \eqref{eqn:Tpdef} and \eqref{eqn:Tp2def}).  We next show that the theta lift commutes with the Hecke operators.
\begin{theorem}\label{thm:Heckepos}

\noindent
\begin{enumerate}
\item
For every weight $k+\frac{1}{2}$ weak Maass form $H$ with eigenvalue $\lambda_s$ with $\re(s)\geq \frac{k}{2}+\frac{1}{4}$
$$
\Phipos(H)\Big|_{2k} T_p=\Phipos\left(H\Big|_{k+\frac{1}{2}} T_{p^2}\right).
$$
\item
If $\re(s)\geq \frac{k}{2}+\frac{1}{4}$ and $s\neq \frac{k}{2}+\frac{1}{4}$, then the lift $\Phipos$ is injective on the space of weak Maass forms with eigenvalue $\lambda_s$ under $\Delta_{k+\frac{1}{2}}$.
\end{enumerate}
\end{theorem}

We next describe a theta lift which parallels the construction of Shintani \cite{Shintani} and Niwa \cite{Niwa} in negative weight.
Define the following theta function
\begin{equation}\label{eqn:thetadef2}
\Thneg(\tauz,\ztau):=v^k \sum_{\substack{D\in \Z\\ Q\in \QD}}Q_{\tauz}Q(\tauz,1)^{k-1} e^{-\frac{4\pi |Q(\tauz,1)|^2 v}{y^2}}  e^{-2\pi i D\ztau}.
\end{equation}
The function $\Thneg$ transforms like a modular form of weight $\frac32-k$ in $\ztau$ and weight $2-2k$ in $\tauz$ (see Proposition \ref{prop:Thetamodular} (2)).  Similar to the positive weight case, for a weak Maass form $H$ of weight $\frac{3}{2}-k$, we define the theta lift by
$$
\Phineg (H)(\tauz):=\left<H,\Thneg\left(-\bar{\tauz},\cdot\right)\right>^{\text{reg}}.
$$
Since the space of weak Maass forms is spanned by the Poincar\'e series  $\PP_{\frac{3}{2}-k,s,D}$ (defined in \eqref{eqn:Psdef}), it suffices to consider their image under the theta lifting.  This leads to the definition 
\begin{equation}\label{eqn:F1-kdef}
\mathcal{F}_{1-k,s,D}\left(\tauz\right) :=\sum_{Q\in \QD}\sgn\left(Q_{\tauz}\right)Q(\tauz,1)^{k-1}\phineg_s\left(\frac{Dy^2}{\left|Q(\tauz,1)\right|^2}\right),
\end{equation}
where, for $0<w\leq 1$ and $s\in \C$ with $\re(s)\geq \frac{k}{2}-\frac{3}{4}$, we define
$$
\phineg_s(w):=\frac{\Gamma\left(s+\frac{k}{2}-\frac14\right)(4\pi D)^{\frac{3}{4}-\frac{k}{2}}}{12\sqrt{\pi} \Gamma(2s)} w^{\frac{k}{2}-\frac{3}{4}+s}{_2F_1}\left(s-\frac{k}{2}+\frac{1}{4}, s+\frac{k}{2}-\frac{3}{4};2s;w\right).
$$
The Euler integral representation \eqref{eqn:Eulerint} again implies that 
$$
\phineg_s(w)=\frac{\Gamma\left(s+\frac{k}{2}-\frac14\right) (4\pi D)^{\frac{3}{4}-\frac{k}{2}}}{12\sqrt{\pi}\Gamma\left(s+\frac{k}{2}-\frac{3}{4}\right)\Gamma\left(s-\frac{k}{2}+\frac{3}{4}\right)}w^{\frac{k}{2}-\frac{3}{4}+s}\int_{0}^{1}t^{s+\frac{k}{2}-\frac{7}{4}}(1-t)^{s-\frac{k}{2}-\frac{1}{4}}\left(1-wt\right)^{-s+\frac{k}{2}-\frac{1}{4}}dt.
$$
  In the special case that $s=\frac{k}{2}+\frac{1}{4}$, a change of variables yields the locally harmonic Maass form
$$
\mathcal{F}_{1-k,D}\left(\tauz\right):=\frac{1}{12\psi(1)}\left(4\pi D\right)^{\frac{3}{4}-\frac{k}{2}}\sum_{Q\in \QD}\sgn\left(Q_{\tauz}\right) Q\left(\tauz,1\right)^{k-1} \psi\left(\frac{Dy^2}{\left|Q\left(\tauz,1\right)\right|^2_{\phantom{-}}}\right),
$$
investigated in \cite{BKW}.  Here
$$
\psi\left(v\right):=\frac{1}{2}\beta\left(v;k-\frac{1}{2},\frac{1}{2}\right)
$$
is a special value of the incomplete $\beta$-function, which is defined for $r,s\in \C$ satisfying $\re\left(r\right)$, $\re\left(s\right)>0$ by $\beta\left(v;s,r\right):=\int_{0}^v u^{s-1}\left(1-u\right)^{r-1}du$.  In \cite{BKW}, the first two authors and Kohnen introduced the functions $\mathcal{F}_{1-k,D}$ and showed that they transform like weight $2-2k$ modular forms and are locally harmonic in every neighborhood of $\H$ which does not intersect $E_D$.  More generally, the functions $\mathcal{F}_{1-k,s,D}$ are local Maass forms with exceptional set $E_D$.
\begin{theorem}\label{thm:negwt}
Suppose that $k$ is even, $D>0$ is a discriminant, and $s\in \C$ satisfies $\re(s)\geq \frac{k}{2}-\frac{3}{4}$.  Then the following hold.

\noindent
\begin{enumerate}
\item
The function $\mathcal{F}_{1-k,s,D}$ is a local Maass form of weight $2-2k$ with eigenvalue $4\lambda_s$ under $\Delta_{2-2k}$ and exceptional set $E_D$.
\item
The theta lift $\Phineg$ maps weight $\frac{3}{2}-k$ weak Maass forms with eigenvalue $\lambda_s$ under $\Delta_{\frac{3}{2}-k}$ to weight $2-2k$ local Maass forms with eigenvalue $4\lambda_s$ under $\Delta_{2-2k}$.  In particular, the image of $\PP_{\frac{3}{2}-k,s,D}$ under the theta lift is
\begin{equation}\label{eqn:thetalift1}
\Phineg\left( \PP_{\frac{3}{2}-k,s,D}\right) =\mathcal{F}_{1-k,s,D}.
\end{equation}
\end{enumerate}
\end{theorem}
\begin{remarks}
\noindent

\noindent
\begin{enumerate}
\item
The functions $\mathcal{F}_{1-k,s,D}$ are never continuous.  That is to say, for every $s$ and $D$ satisfying the conditions of Theorem \ref{thm:negwt}, there exist points along $E_D$ for which $\mathcal{F}_{1-k,s,D}$ exhibits discontinuities.  
\item
Although $\mathcal{F}_{1-k,D}$ is never continuous, one may add a piecewise polynomial function to obtain a real analytic function.  The polynomial in question is related to the period polynomial of $f_{k,D}$ and was thoroughly investigated in \cite{BKW}. 
\item
In the omitted case $k=1$ and $\lambda_s=0$, H\"ovel \cite{Hovel} has constructed locally harmonic Maass forms via a theta lift.  The relationship with the Shimura and Shintani lifts as well as its geometric interpretation were further investigated there.  
\item
The regularized theta lifts considered here should also have a geometric interpretation.  One expects that their images represent cohomology classes of geodesic cycles as currents.
\end{enumerate}
\end{remarks}
We again turn to the general properties of this theta lift.  In particular, it also commutes with the Hecke operators.
\begin{theorem}\label{thm:Heckeneg}
Suppose that $s\in \C$ satisfies $\re(s)\geq \frac{k}{2}-\frac{3}{4}$.
The following hold.

\noindent
\begin{enumerate}
\item
For every weak Maass form $H$ of weight $\frac{3}{2}-k$ in Kohnen's plus space with eigenvalue $\lambda_s$ under $\Delta_{\frac{3}{2}-k}$, one has
\begin{equation}\label{eqn:Hecke}
\Phineg(H)\Big|_{2-2k} T_p=\Phineg\left(H\Big|_{\frac32-k} T_{p^2}\right).
\end{equation}
\item
The lift $\Phineg$ is injective on the space of weak Maass forms with eigenvalue $\lambda_s$ under $\Delta_{\frac{3}{2}-k}$.
\end{enumerate}
\end{theorem}
\begin{remark}
In \cite{BKW}, it was shown that the functions $\mathcal{F}_{1-k,D}$ satisfy relations under the Hecke operators which seemed to imply a natural connection to weight $\frac{3}{2}-k$ objects.  This is explained by the relation \eqref{eqn:Hecke} between integral and half-integral weight Hecke operators.
\end{remark}

The images $\mathcal{F}_{1-k,s,D}$ and $f_{k,s,D}$ under the two theta lifts considered in this paper are related through the antiholomorphic differential operator $\xi_{\wt}:=2iy^{\wt}\overline{\frac{\partial}{\partial\overline{\tauz}}}$.
\begin{theorem}\label{thm:negposrel}
Suppose that $k>0$ is an even integer, $D$ is a positive discriminant, and $s\in \C$ satisfies $\re(s)\geq \frac{k}{2}+\frac{1}{4}$.  

\noindent
\begin{enumerate}
\item
For every $\tauz\notin E_D$, we have that
\begin{equation}\label{eqn:xiFs}
\xi_{2-2k}\left(\mathcal{F}_{1-k,s,D}(\tauz)\right)=  2\left(\overline{s}-\frac{3}{4}+\frac{k}{2}\right) f_{k,\overline{s},D}(\tauz).
\end{equation}
\item
For $\tauz\notin E_D$, we have that
\begin{equation}\label{eqn:xi2k}
\xi_{2k}\left(f_{k,s,D}(\tauz)\right)= 2\left(\overline{s}-\frac{k}{2}-\frac{1}{4}\right)\mathcal{F}_{1-k,\overline{s},D}(\tauz).
\end{equation}
\end{enumerate}
\end{theorem}
Theorem \ref{thm:negposrel} states that for $s\geq \frac{k}{2}+\frac{1}{4}$ the following commutative diagram holds:
$$
\xymatrix{\PP_{\frac{3}{2}-k,s,D}\ar@{->}[dd]_{\xi_{\frac{3}{2}-k}}\ar@{->}[rrr]^{\Phineg}& && \mathcal{F}_{1-k,s,D}\ar@{->}[dd]^{\xi_{2-2k}}\\
\\
\left(\overline{s}-\frac{3}{4}+\frac{k}{2}\right) \PP_{k+\frac{1}{2},\overline{s},D}\ar@{->}[rrr]^{2\Phipos}
\ar@{->}[dd]^{\xi_{k+\frac{1}{2}}} &&& 2\left(\overline{s}-\frac{3}{4}+\frac{k}{2}\right) f_{k,\overline{s},D}\ar@{->}[dd]_{\xi_{2k}} \\
\\
-\lambda_s \PP_{\frac{3}{2}-k,s,D}\ar@{->}[rrr]^{4\Phineg} &&& -4\lambda_s \mathcal{F}_{1-k,s,D}
}
$$
Denote the $d$-th Shimura \cite{Shimura} lift by $\mathscr{S}_d$ and $\PP_{\wt,D}:=\PP_{\wt,\frac{k}{2}+\frac{1}{4},D}$.  In the special case that $s=\frac{k}{2}+\frac{1}{4}$ (see Corollary 9 of \cite{KohnenZagier} for the constant multiple of $\mathscr{S}_1$), the diagram becomes the following:
$$
\xymatrix{
\PP_{\frac{3}{2}-k,D}\ar@{->}[dd]_{\xi_{\frac{3}{2}-k}}\ar@{->}[rrr]^{\Phineg} &&&
\mathcal{F}_{1-k,D}\ar@{->}[dd]^{\xi_{2-2k}}\\
\\
\left(k-\frac{1}{2}\right)\PP_{k+\frac{1}{2},D}\ar@{->}[rrr]^{2\Phipos}_{3^{-1}2^{-k}\mathscr{S}_1} &&&
\frac{2^{2k-3}}{3}\left(4\pi D\right)^{\frac{3}{4}-\frac{k}{2}}f_{k,D}
}
$$
\begin{remarks}
\noindent
\begin{enumerate}
\item
The above diagram extends work of Bruinier and Funke \cite{BF} and H\"ovel \cite{Hovel} in the case of $O(2,1)$ to higher weight.
\item
By applying \eqref{eqn:xiFs2} (used to obtain \eqref{eqn:xiFs}) to $s$-derivatives of weak Maass forms, one could also obtain links between modular objects known as sesquiharmonic forms \cite{BDR}.  These functions map to weakly holomorphic modular forms under the hyperbolic Laplacian.
\end{enumerate}
\end{remarks}

The paper is organized as follows.  In Section \ref{sec:harmonic}, we recall the theory of weak Maass forms and give a formal definition of local Maass forms.  Section \ref{sec:regularized} is devoted to the properties of the regularized inner product.  The modularity properties of the theta functions are enunciated in  Section \ref{sec:theta}, where we derive a number of interrelations between the theta functions through differential operators.  The image of $\Phipos$ (Theorem \ref{thm:poswt} (2)) is determined in Section \ref{sec:Phikimage}, while Section \ref{sec:Phi1-kimage} is devoted to the image of $\Phineg$ (Theorem \ref{thm:negwt} (2)) and the injectivity of the lift (Theorem \ref{thm:Heckeneg} (2)).  In Section \ref{sec:negposwt}, Theorem \ref{thm:negposrel} is established and the relationship between $f_{k,s,D}$ and $\mathcal{F}_{1-k,s,D}$ is then used to conclude Theorems \ref{thm:poswt} (1) and \ref{thm:negwt} (1).  Finally, Section \ref{sec:Hecke} concludes the paper with a discussion of the Hecke operators and the injectivity of $\Phipos$ (Theorems \ref{thm:Heckepos} and \ref{thm:Heckeneg} (1)).
\section*{Acknowledgements}
The authors thank Jan Bruinier for suggesting to investigate the connection between $\mathcal{F}_{1-k,D}$ and theta lifts and for fruitful discussion.  The authors also thank Jens Funke for helpful comments which aided the exposition.

\section{Basic facts on weak and local Maass forms}\label{sec:harmonic}

In this section, we recall the basic definitions necessary to describe the modular objects and the theta lifts used in this paper.  We first define the regularized inner product used in the definitions of $\Phiposno$ and $\Phinegno$.  In order to understand the relationship between lifts in different spaces, we then define the Hecke operators, which act formally on any translation invariant function.  We then recall Kohnen's plus space and weak Maass forms, upon which we apply our theta lifts.  The next subsection is devoted to constructing Poincar\'e series which span these spaces of weak Maass forms.  Following this, we give the definition of local Maass forms, which are the focus of this paper. 

Thoughout this section, $\wt\in\frac{1}{2}\Z$ and we set $\Gamma:=\SL_2(\Z)$ whenever $\wt\in \Z$, while $\Gamma:=\Gamma_0(4)$ if $\wt\in \frac{1}{2}\Z\setminus\Z$.  
\subsection{Regularized inner products and Hecke operators}\label{sec:regularized}
For $\bigT>0$, denote the truncated fundamental domain for $\SL_2(\Z)$ by
\begin{equation}\label{eqn:truncdef}
\mathcal{F}_{\bigT} := \left\{ \ztau\in \H: |u|\leq \frac{1}{2}, |\ztau|\geq 1, v\leq \bigT\right\}.
\end{equation}
For a finite index subgroup $\Gamma\subseteq\SL_2(\Z)$ we further define
$$
\mathcal{F}_{\bigT}\left(\Gamma\right):=\bigcup_{\gamma\in \Gamma\backslash \SL_2(\Z)} \gamma \mathcal{F}_{\bigT}.
$$
In particular, we set $\mathcal{F}_{\bigT}(4):=\mathcal{F}_{\bigT}\left(\Gamma_0(4)\right)$.
For two functions $G$ and $H$ satisfying weight $\wt$ modularity for the group $\Gamma$, we define, whenever the limit exists, the regularized inner product 
$$
\left<G,H\right>^{\text{reg}}:=\frac{1}{\left[\SL_2(\Z):\Gamma\right]}\lim_{\bigT\to\infty}\int_{\mathcal{F}_{\bigT}(\Gamma)} G(\ztau)\overline{H(\ztau)} v^{\wt} \frac{dudv}{v^2}.
$$
We use the following lemma, which follows by standard arguments using Stokes's Theorem.
\begin{lemma}\label{lem:Stokes}
Suppose that $F$, $G:\mathbb{H}\to\mathbb{C}$ are real analytic functions that satisfy $F|_{2-\wt} \gamma=F$ and $G|_{\wt} \gamma=G$ for all $\gamma\in\Gamma$.  Then
\begin{equation}\label{eqn:xicommute prepare}
\int_{\mathcal{F}_{\bigT}\left(\Gamma\right)} \xi_{2-\wt}\left(F(\ztau)\right)\,\overline{G(\ztau)}\,v^{\wt-2}\,du\,dv+\int_{\mathcal{F}_{\bigT}\left(\Gamma\right)} \xi_{\wt}\left(G(\ztau)\right)\,\overline{F(\ztau)}\,v^{-\wt}\,du\,dv=-\int_{\partial\mathcal{F}_{\bigT}\left(\Gamma\right)} \overline{F(\ztau)G(\ztau)}\,d\overline{\ztau}.
\end{equation}
\end{lemma}

A number of important operators are Hermitian with respect to the regularized inner product.  One such class of operators is the \begin{it}Hecke operators.\end{it}  Suppose that $F$ is a function satisfying weight $\wt$ modularity and write its Fourier expansion as 
$$
F(\ztau)=\sum_{n\in\Z} a_{v}(n)e^{2\pi i nu}.
$$
If $\wt\in\Z$ (resp. $\wt\in \frac{1}{2}\Z\setminus\Z$), then for a prime $p$, the Hecke operator $T_p$ (resp. $T_{p^2}$) is defined by
\begin{align}\label{eqn:Tpdef}
F\Big|_{\wt}T_p(\ztau)&:=\sum_{n\in \Z} \left(a_{v}(pn)+p^{\wt-1}a_{v}\left(\frac{n}{p}\right) \right)e^{2\pi i nu},\\
\label{eqn:Tp2def}
F\Big|_{\wt}T_{p^2}(\ztau)&:=\sum_{n\in \Z} \left( a_{v}\left(p^2n\right)+ p^{\wt-\frac{3}{2}}\left(\frac{(-1)^{\wt-\frac{1}{2}}n}{p}\right)a_v(n)+ p^{2\wt-2} a_{v}\left(\frac{n}{p^2}\right)\right)e^{2\pi i nu}.
\end{align}

We apply the regularized inner product to (half-integral weight) weak Maass forms, which we define in the following subsection.
\subsection{Weak Maass forms}
When $\wt\in \frac{1}{2}\Z\setminus\Z$, we are interested in weight $\wt$ real analytic modular forms on $\Gamma$ in Kohnen's plus space.  This means that the Fourier expansions are supported on the coefficients $n$ satisfying $(-1)^{\wt-\frac{1}{2}}n\equiv 0,1\pmod{4}$.  We use $\pr$ to denote the projection operator (see Section 2.3 of \cite{KohnenMathAnn}) into Kohnen's plus space.  It is useful to recall that if $F$ is modular in Kohnen's plus space for $\Gamma$, then its Fourier expansions at the cusps $0$ and $\frac{1}{2}$ are determined by the expansion at $i\infty$ (see \cite{KohnenMathAnn} for a proof in the holomorphic case).  Like the Hecke operators, the projection operator $\pr$ is Hermitian with respect to the regularized inner product, i.e.,
\begin{equation}\label{eqn:prherm}
\left<G\Big|\pr, H\right>^{\text{reg}} = \left<G,H\Big|\pr\right>^{\text{reg}}.
\end{equation}

The real analytic modular forms of particular interest for this paper are weak Maass forms.  A good background reference for weak Maass forms is \cite{BF}.  Recall that we write $\ztau=u+iv$ throughout.  For $\wt\in \frac{1}{2}\Z$, the weight $\wt$ \begin{it}hyperbolic Laplacian\end{it} is defined by
$$
\Delta_{\wt}:=\Delta_{\wt,\ztau} := -v^2\left( \frac{\partial^2}{\partial u^2}+\frac{\partial^2}{\partial v^2}\right) + i\wt v\left(\frac{\partial}{\partial u}+i \frac{\partial}{\partial v}\right).
$$
It is related to the operator $\xi_{\wt}=\xi_{\wt,\ztau}:=2iv^{\wt} \overline{\frac{\partial}{\partial\overline{\ztau}}}$ through
$$
\Delta_{\wt} = -\xi_{2-\wt}\circ\xi_{\wt}.
$$
In order to define weak Maass forms, we require
\begin{equation}\label{eqn:Mdef}
\mathcal{M}_{\wt,s}\left(t\right):=\left|t\right|^{-\frac{\wt}{2}}M_{\frac{\wt}{2}\sgn(t),\, s-\frac{1}{2}}\left(|t|\right),
\end{equation}
where $M_{\mu,s-\frac{1}{2}}$ is the usual $M$-Whittaker function.   For $\re\left(s\pm \mu \right)>0$ and $v>0$, we have the integral representation
\begin{equation}\label{eqn:Mintrep}
M_{\mu,s-\frac{1}{2}}(v)=v^{s}e^{\frac{v}{2}}\frac{\Gamma(2s)}{\Gamma\left(s+\mu\right)\Gamma\left(s-\mu\right)}\int_0^1 t^{s+\mu-1}(1-t)^{s-\mu-1}e^{-vt}dt.
\end{equation}
In the special case that $\mu=s$, we have
\begin{equation}\label{eqn:M1F1}
M_{\mu,s-\frac{1}{2}}(v)=e^{-\frac{v}{2}} v^s.
\end{equation}
Furthermore, as $v\to\infty$, the Whittaker function satisfies the following asymptotic behavior for $\mu\neq s$:
\begin{equation}\label{eqn:Whittasympt}
M_{\mu,s-\frac{1}{2}}(v)\sim \frac{\Gamma(2s)}{\Gamma(s-\mu)} e^{\frac{v}{2}} v^{-\mu}.
\end{equation}

We move on to the definition of weak Maass forms.  
For $s\in \C$ a \begin{it}weak Maass form\end{it} $F:\H\to\C$ of weight $\wt$ for $\Gamma$ with eigenvalue $\lambda=\left(s-\frac{\wt}{2}\right)\left(1-s-\frac{\wt}{2}\right)$ is a real analytic function satisfying:
\noindent

\noindent
\begin{enumerate}
\item 
For every $\gamma\in \Gamma$, one has $F|_{\wt} \gamma = F$, where $|_{\wt}$ denotes the usual weight $\wt$ slash-operator.
\item 
One has 
$\Delta_{\wt}\left(F\right)=\lambda F.$
\item
 There exist $a_1,\dots,a_N\in \C$ for which
$$
F(\ztau)-\sum_{m=1}^N a_m\mathcal{M}_{\wt,s}\left(4\pi \sgn(\wt)m v\right)e^{2\pi i m\sgn(\wt)u}=O\left(v^{1-\re(s)-\frac{\wt}{2}}\right).
$$
There are analogous conditions at the other cusps of $\Gamma$.
\end{enumerate}
\subsection{Poincar\'e series}
One builds explicit examples of weak Maass forms by constructing Poincar\'e series \cite{Fay}.  For $m\in\Z\setminus\{0\}$, the function
$$
\psi_{m,\wt}\left(s;\ztau\right):=\left(4\pi |m|\right)^{\frac{\wt}{2}}\Gamma(2s)^{-1}\mathcal{M}_{\wt,s}\left(4\pi m v\right)e^{2\pi i mu}
$$
is an eigenfunction for $\Delta_{\wt}$ with eigenvalue $\left(s-\frac{\wt}{2}\right)\left(1-s-\frac{\wt}{2}\right)$.  Thus, one concludes that for $\re(s)>1$ the Poincar\'e series
\begin{equation}\label{eqn:Poincdef}
\PP_{\wt,s,\Gamma,m}(\ztau):=\sum_{\gamma\in \Gamma_{\infty}\backslash \Gamma} \psi_{\sgn(\wt)m,\wt}\left(s;\ztau\right)\Big|_{\wt}\gamma,
\end{equation}
where $\Gamma_\infty:=\left\{\pm\left(\begin{smallmatrix} 1&n\\0&1\end{smallmatrix}\right):n\in\Z\right\}$, 
is also an eigenfunction under $\Delta_{\wt}$ with the same eigenvalue.  Moreover, the space of weight $\wt$ weak Maass forms with this eigenvalue is spanned by such Poincar\'e series. The Poincar\'e series satisfies the growth condition
\begin{equation}\label{eqn:Poincgrowth}
\PP_{\wt,s,\Gamma,m}(\ztau)-\psi_{\sgn(\wt)m,\wt}\left(s;\ztau\right) = O\left(v^{1-\re(s)-\frac{\wt}{2}}\right).
\end{equation}
In the case that $\wt\in \frac{1}{2}\Z\setminus\Z$, we then project the Poincar\'e series into Kohnen's plus space, defining
\begin{equation}\label{eqn:Psdef}
\PP_{\wt,s,m}:=\PP_{\wt,s,\Gamma_0\left(4\right),m}\Big|\pr.
\end{equation}
In the special cases that $s=1-\frac{\wt}{2}$ or $s=\frac{\wt}{2}$, the resulting Poincar\'e series is harmonic.  For $D\neq 0$, the positive and negative weight Poincar\'e series are related to each other via
\begin{equation}\label{eqn:xiPs}
\xi_{\wt}\left(\PP_{\wt,s,D}\right)=\left(\overline{s}-\frac{\wt}{2}\right) \PP_{2-\wt,\overline{s},D}.
\end{equation}

\subsection{Local Maass forms}
Mirroring the definition of weak Maass forms, for $\wt\in 2\Z$, $\lambda\in \C$, and a measure zero set $E$, we call a function $\mathcal{F}$ a weight $\wt$ \begin{it}local Maass form\end{it} with eigenvalue $\lambda$ and exceptional set $E$ if $\mathcal{F}$ satisfies the following:
\noindent
\begin{enumerate}
\item
For every $\gamma\in \SL_2(\Z)$, one has $\mathcal{F}|_{\wt}\gamma = \mathcal{F}$ 
\item
For every $\ztau\notin E$ there exists a neighborhood around $\ztau$ for which $\mathcal{F}$ is real analytic and
$$
\Delta_{\wt}(\mathcal{F})(\ztau) = \lambda \mathcal{F}(\ztau).
$$
\item
For $\ztau\in E$ one has
$$
\mathcal{F}(\ztau) = \frac{1}{2}\lim_{r\to 0^+} \left(\mathcal{F}\left(\ztau+ir\right)+\mathcal{F}\left(\ztau-ir\right)\right).
$$
\item
The function $\mathcal{F}$ exhibits at most polynomial growth as $v\to \infty$.
\end{enumerate}
Examples of locally harmonic Maass forms (those with eigenvalue $0$) are given in \cite{BKW} as ``quadratic form Poincar\'e series.''  In this paper, we give further examples of local Maass forms via theta lifts.

\section{Indefinite theta functions}\label{sec:theta}

In this section we collect several important properties of the theta functions \eqref{eqn:thetadef1} and \eqref{eqn:thetadef2}.  The modularity properties of these indefinite theta functions follow by a result of Vign\'eras \cite{Vigneras}.  To state these, we define the Euler operator $E:=\sum_{i=1}^{n} w_i\frac{\partial}{\partial w_i}$.  As usual, we denote the Gram matrix associated to a nondegenerate quadratic form $q$ on $\R^n$ by $A$.  The \begin{it}Laplacian\end{it} associated to $q$ is then defined by $\Delta:=\left\langle \frac{\partial}{\partial w}, A^{-1}\frac{\partial}{\partial w}\right\rangle$.  Here $\left<\cdot,\cdot\right>$ denotes the usual inner product on $\R^n$.

\begin{theorem}[Vign\'eras]\label{thm:Vigneras}
Suppose that $n\in \N$, $q$ is a nondegenerate quadratic form on $\R^n$, $L\subset \R^n$ is a lattice on which $q$ takes integer values, and $p:\R^n\to\C$ is a function satisfying the following conditions:
\begin{itemize}
\item[(i)]
The function $f(w):=p(w)e^{-2\pi q(w)}$ times any polynomial of degree at most 2 and all partial derivatives of $f$ of order at most $2$ are elements of $L^2\left(\R^n\right)\cap L^1\left(\R^n\right)$.
\item[(ii)]
For some $\lambda\in\Z$, the function $p$ satisfies
\[
\left(E-\frac{\Delta }{4\pi}\right)p=\lambda p.
\]
\end{itemize}
Then the indefinite theta function
\[
v^{-\frac{\lambda}{2}}\sum_{w\in L} p\left(w\sqrt{v}\right) e^{2\pi i q(w)\ztau}
\]
is modular of  weight $\lambda+\frac{n}{2}$ for $\Gamma_0(N)$ and character $\chi\cdot\chi_{-4}^{\lambda}$, where $N$ and $\chi$ are the level and character of $q$ and $\chi_{-4}$ is the unique primitive Dirichlet character of conductor $4$.
\end{theorem}
\begin{remark}
Note that the definition of the character given in Vign\'eras \cite{Vigneras} differs to that given by Shimura \cite{Shimura} by a factor of $\chi_{-4}^{\lambda}$.  We adopt Shimura's notation here.
\end{remark}
Applying Theorem \ref{thm:Vigneras} to $\Thpos$ and $\Thneg$  yields their modularity properties (see \cite{BKZ} for details). 
\begin{proposition}\label{prop:Thetamodular}
\noindent

\noindent
\begin{enumerate}
\item
The function $\Thpos\left(-\overline{\tauz},\ztau\right)$ transforms like a modular form of weight $k+\frac{1}{2}$ in Kohnen's plus space on $\Gamma_0(4)$ in $\ztau$ and weight $2k$ on $\SL_2(\Z)$ in $\tauz$.
\item
The function $\Thneg$ transforms like a modular form of weight $\frac32-k$ in Kohnen's plus space on $\Gamma_0(4)$ in $\ztau$ and weight $2-2k$ on $\SL_2(\Z)$ in $\tauz$.
\end{enumerate}
\end{proposition}

The following lemma is the key relation needed to establish a link between the functions $f_{k,s,D}$ and $\mathcal{F}_{1-k,s,D}$.  The correspondence is formed through a relation betwen the respective differential operators in $\ztau$ and $\tauz$ on $\Thpos$ and $\Thneg$, mirroring an important connection formed in \cite{BF}.
\begin{lemma}\label{lem:xi1}
For every integer $k\geq 1$, one has
\begin{align}\label{eqn:thetaid1}
\xi_{k+\frac{1}{2},\ztau}\left(\Thpos\left(\tauz,\ztau\right)\right)&=-iy^{2-2k}\frac{\partial}{\partial\tauz}\Thneg(-\overline{\tauz},\ztau),\\
\label{eqn:xiThtauTh*z}
\xi_{\frac{3}{2}-k,\ztau}\left(\Thneg\left(-\overline{\tauz},\ztau\right)\right)&=-i y^{2k}\frac{\partial}{\partial \tauz}\Thpos\left(\tauz,\ztau\right).
\end{align}
\end{lemma}

\begin{proof}
We first prove \eqref{eqn:thetaid1}.  We compute that $\frac{\partial}{\partial \tauz}\Thneg\left(-\overline{\tauz},\ztau\right)$ equals
$$
\sum_{\substack{D\in \Z \\ Q\in \QD}} Q(-\overline{\tauz},1)^{k-1}e^{-\frac{4\pi v}{y^2} \left|Q(-\overline{\tauz},1)\right|^2} e^{-2\pi i D\tau}\left(\frac{\partial}{\partial \tauz}Q_{-\overline{\tauz}} - 4 \pi Q_{-\overline{\tauz}} v\frac{\partial}{\partial \tauz}\left(\frac{|Q(-\overline{\tauz},1)|^2}{y^2}\right)\right).
$$
We then use
\begin{equation}\label{eqn:Qrewrite}
\left|Q(\tauz,1)\right|^2 = Q_{\tauz}^2y^2 + Dy^2
\end{equation}
and 
$$
y^2\frac{\partial}{\partial \tauz}Q_{-\overline{\tauz}} = \frac{i}{2} Q(-\overline{\tauz},1)
$$
to obtain
$$
-iy^{2-2k}\frac{\partial}{\partial \tauz}\Thneg\left(-\overline{\tauz},\ztau\right)= \frac{1}{2}y^{-2k}v^k \sum_{\substack{D\in \Z \\ Q\in \QD}} Q(-\overline{\tauz},1)^{k}e^{-4\pi Q_{-\overline{\tauz}}^2 v}e^{-2\pi i D\overline{\ztau}}\left(1- 8 \pi Q_{-\overline{\tauz}}^2 v\right).
$$

We similarly compute the action of $\xi_{k+\frac{1}{2},\ztau}$ on $\Thpos$.  A straightforward calculation yields 
$$
\xi_{k+\frac{1}{2},\ztau}\left(\Thpos\left(\tauz,\ztau\right)\right)= \frac{1}{2}y^{-2k}v^k \sum_{\substack{D\in \Z\\ Q\in \QD}} Q(\overline{\tauz},1)^{k} e^{-4\pi Q_{\tauz}^2v}e^{-2\pi i D\overline{\ztau}}\left(1 - 8\pi Q_{\tauz}^2 v\right).
$$
Equation \eqref{eqn:thetaid1} now follows immediately by the change of variables $Q=[a,b,c]\to [a,-b,c]=:\widetilde{Q}\in \QD$, noting that
\begin{equation}\label{eqn:Qtilderef}
\widetilde{Q}\left(\overline{z},1\right)=Q\left(-\overline{z},1\right)\qquad \text{and}\qquad \widetilde{Q}_{\tauz}=Q_{-\overline{\tauz}}.
\end{equation}

We move on to proving \eqref{eqn:xiThtauTh*z}.  Since $Q_{\tauz}\in \R$, a direct calculation, mirroring the proof of \eqref{eqn:thetaid1} and using \eqref{eqn:Qtilderef}, yields
$$
\xi_{\frac{3}{2}-k,\ztau}\left(\Thneg(-\overline{\tauz},\ztau)\right) = v^{\frac{1}{2}} \sum_{\substack{D\in \Z\\ Q\in \QD}} Q_{\tauz}Q(\tauz,1)^{k-1}e^{-4\pi Q_{\tauz}^2 v}e^{2\pi i D\ztau}\left( k - \frac{4\pi v}{y^2}\left|Q(\tauz,1)\right|^2\right).
$$
We next obtain \eqref{eqn:xiThtauTh*z} by showing that $-i y^{2k}\frac{\partial}{\partial \tauz}\Thpos\left(\tauz,\ztau\right)$ equals
\begin{multline*}
-iv^{\frac{1}{2}}y^{2}\sum_{\substack{D\in \Z\\ Q\in \QD}} Q(\tauz,1)^{k-1} e^{-4\pi Q_{\tauz}^2v} e^{2\pi i D\ztau} \left(k\frac{\partial}{\partial \tauz} \left(y^{-2}Q\left(\tauz,1\right)\right)  -8\pi Q_{\tauz} y^{-2}Q\left(\tauz,1\right) v\frac{\partial}{\partial \tauz}Q_{\tauz}\right)\\
=v^{\frac{1}{2}} \sum_{\substack{D\in \Z\\ Q\in \QD}}Q_{\tauz} Q(\tauz,1)^{k-1} e^{-4\pi Q_{\tauz}^2v} e^{2\pi i D\ztau}\left(k-\frac{4\pi v}{y^2}\left|Q(\tauz,1)\right|^2\right),
\end{multline*}
where in the last line we have used 
$$
y^2\frac{\partial}{\partial \tauz}Q_{\tauz} = \frac{i}{2} Q(\overline{\tauz},1)\qquad \text{and}\qquad y^2\frac{\partial}{\partial \tauz}\left(y^{-2}Q\left(\tauz,1\right)\right) =i Q_{\tauz}.
$$
\end{proof}

The following lemma relates the regularized inner products in positive and negative weight through the $\xi$-operator.
\begin{lemma}\label{lem:xiTh}
Suppose that $D>0$ is a discriminant and $\tauz\notin E_D$.  Then for every $s$ with $\re(s)\geq \frac{k}{2}+\frac{1}{4}$ one has
\begin{equation}\label{eqn:xicommuteTh}
\left< \xi_{k+\frac{1}{2}}\left(\PP_{k+\frac{1}{2},s,D}\right), \Thneg\left(-\overline{\tauz},\cdot\right) \right>^{\text{reg}} = -\overline{\left< \PP_{k+\frac{1}{2},s,D}, \xi_{\frac{3}{2}-k}\left(\Thneg\left(-\overline{\tauz},\cdot\right)\right)\right>^{\text{reg}}}
\end{equation}
and
\begin{equation}\label{eqn:xicommuteTh*}
\left< \xi_{\frac{3}{2}-k}\left(\PP_{\frac{3}{2}-k,s,D}\right), \Thpos(\tauz,\cdot) \right>^{\text{reg}} = -\overline{\left< \PP_{\frac{3}{2}-k,s,D}, \xi_{k+\frac{1}{2}}\left(\Thpos(\tauz,\cdot)\right)\right>^{\text{reg}}}.
\end{equation}
\end{lemma}
\begin{proof}
Note that all of the regularized integrals exist, as will be shown in the proofs of Theorem \ref{thm:poswt} (2) and \ref{thm:negwt} (2). We begin with the proof of \eqref{eqn:xicommuteTh} and abbreviate $P:=\PP_{k+\frac{1}{2},s,D}$.  By Lemma \ref{lem:Stokes}, we have
$$
\left< \xi_{k+\frac{1}{2}}\left(P\right), \Thneg\left(-\overline{\tauz},\cdot\right) \right>^{\text{reg}}+ \overline{\left< P, \xi_{\frac{3}{2}-k}\left(\Thneg\left(-\overline{\tauz},\cdot\right)\right)\right>^{\text{reg}}}=-\frac{1}{6}\lim_{\bigT\to\infty} \int_{\partial \mathcal{F}_{\bigT}(4)} \overline{P(\ztau)\Thneg\left(-\overline{\tauz},\ztau\right)}d\overline{\ztau},
$$
provided that the limit exists.  Hence our goal is to show that the limit on the right hand side is zero.  
A standard argument reduces this claim to showing that
\begin{equation}\label{eqn:StokesRemainTh}
\lim_{\bigT\to\infty} \int_{0}^{1} P(u+i\bigT)\Thneg\left(-\overline{\tauz},u+i\bigT\right) du =0
\end{equation}
as well as vanishing of similar integrals around the other cusps of $\Gamma_0(4)$.  However, since both $P$ and $\Thneg$ are in Kohnen's plus space, the vanishing of the corresponding integrals at the other cusps may be reduced to showing that \eqref{eqn:StokesRemainTh} vanishes.

In order to prove \eqref{eqn:StokesRemainTh}, we first recall the growth condition \eqref{eqn:Poincgrowth} and note that $\Thneg\left(-\overline{\tauz},u+i\bigT\right)$ decays exponentially as $\bigT\to \infty$.  Indeed, using \eqref{eqn:Qrewrite}, one can show that for fixed $\tauz\in \H$ the quadratic form
$$
Q^*(a,b,c):=D-\frac{2|Q(\tauz,1)|^2}{y^2}=-D+2Q_{\tauz}^2
$$
 is positive definite on the lattice of all binary quadratic forms $Q=[a,b,c]\in \QD$.  After evaluating the integral over $u$, one reduces \eqref{eqn:StokesRemainTh} to showing that
$$
\lim_{\bigT\to \infty} R_{\bigT}=0,
$$
where 
$$
R_{\bigT}:=\mathcal{M}_{k+\frac{1}{2},s}\left(4\pi D\bigT \right)\bigT^k\sum_{Q\in \QD} Q_{-\overline{\tauz}}Q(-\overline{\tauz},1)^{k-1}e^{-\frac{4\pi\left|Q\left(-\overline{\tauz},1\right)\right|^2 \bigT }{y^2}}e^{2\pi D\bigT}.
$$
However, the asymptotic behavior  for the Whittaker function coming from \eqref{eqn:M1F1} and \eqref{eqn:Whittasympt} yields
$$
\mathcal{M}_{k+\frac{1}{2},s}\left(4\pi D\bigT\right)\ll_{k,s,D} e^{2\pi D\bigT} \bigT^{-k-\frac{1}{2}}.
$$
Using \eqref{eqn:Qrewrite}, we may hence bound
$$
R_{\bigT} \ll_{k,s,D} \bigT^{-\frac{1}{2}} \sum_{Q\in \QD} Q_{-\overline{\tauz}}Q(-\overline{\tauz},1)^{k-1}e^{-4\pi Q_{-\overline{\tauz}}^2 \bigT}.
$$
Since $\tauz\notin E_D$ (and hence $-\overline{\tauz}\notin E_D$), $Q_{-\overline{\tauz}}^2>0$ for every $Q\in \QD$ and hence $R_{\bigT}$ exhibits exponential decay as $\bigT\to \infty$.  This concludes \eqref{eqn:StokesRemainTh}, yielding \eqref{eqn:xicommuteTh}.  The proof of \eqref{eqn:xicommuteTh*} follows analogously. 

\end{proof}

\section{Image of the theta lift $\Phipos$}\label{sec:Phikimage}
In this section, we introduce a spectral parameter in the classical Shintani lift.

\begin{proof}[Proof of Theorem \ref{thm:poswt} (2)]
In order to compute the regularized inner product, we use a method of Zagier \cite{ZagierND}.  He defined a regularization which he used for functions which grow at most polynomially, but the method may be extended to the functions of interest here, as we now describe.  We first define
$$
\H_{\bigT}:=\bigcup_{\gamma\in \SL_2(\Z)}\gamma\mathcal{F}_{\bigT}=\bigcup_{\gamma\in \Gamma_0(4)}\gamma\mathcal{F}_{\bigT}(4).
$$
We first use \eqref{eqn:prherm} together with the fact that $\Thpos=\Thpos|\pr$ to compute
\begin{equation}\label{eqn:Poincherm}
\left<\PP_{k+\frac{1}{2},s,D},\Thpos(\tauz,\cdot)\right>^{\text{reg}} = \left<\PP_{k+\frac{1}{2},s,\Gamma_0(4),D}\Big|\pr,\Thpos(\tauz,\cdot)\right>^{\text{reg}}  = \left<\PP_{k+\frac{1}{2},s,\Gamma_0(4),D},\Thpos(\tauz,\cdot)\right>^{\text{reg}}.
\end{equation}
Then the usual unfolding argument yields
$$
\left<\PP_{k+\frac{1}{2},s,D},\Thpos(\tauz,\cdot)\right>^{\text{reg}} =\frac{1}{6}\lim_{\bigT\to\infty}\int_{\Gamma_{\infty}\backslash\H_{\bigT}} \psi_{D,k+\frac{1}{2}}\left(s;\ztau\right)  \overline{\Thpos\left(\tauz,\ztau\right)}v^{k+\frac{1}{2}}\frac{dudv}{v^2}.
$$
We now rewrite
\[
\H_{\bigT}=\left\{\ztau\in\field{H}\Big|\text{Im}(\ztau)\leq \bigT\right\}\Big\backslash\mathop{\bigcup}_{\substack {c\geq 1\\ a\in\Z\\ (a, c)=1}} S_{\frac{a}{c}}(\bigT),
\]
where $S_{\frac{a}{c}}(\bigT)$ is the disc of radius $\frac{1}{2c^2\bigT}$ tangent to the real axis at $\frac{a}{c}$.  Hence, we have
$$
\left<\PP_{k+\frac{1}{2},s,D},\Thpos(\tauz,\cdot)\right>^{\text{reg}} =\lim_{\bigT\to\infty}\left(I_1(\bigT) + I_2(\bigT)\right),
$$
where
\begin{align*}
I_1(\bigT)&:=\frac{1}{6}\int_0^{\bigT} \int_0^1 \psi_{D,k+\frac{1}{2}}\left(s;\ztau\right) \overline{\Thpos\left(\tauz,\ztau\right)}v^{k+\frac{1}{2}}\frac{dudv}{v^2},\\
I_2(\bigT)&:=-\frac{1}{6}\sum_{c\geq 1}\sum_{\substack{a\; (\text{mod} \; c)\\ (a,c)=1}}\int_{S_{\frac{a}{c}}(\bigT)}\psi_{D,k+\frac{1}{2}}\left(s;\ztau\right)\overline{\Thpos\left(\tauz,\ztau\right)} v^{k+\frac{1}{2}}\frac{dudv}{v^2}.
\end{align*}
We first consider $I_1(\bigT)$.  Evaluating the integral over $u$ and using \eqref{eqn:Qrewrite}, we obtain
\begin{multline}\label{eqn:unfoldfkDs}
\lim_{\bigT\to\infty} I_1(\bigT)=\frac{\left(4\pi D\right)^{\frac{k}{2}+\frac{1}{4}}}{6y^{2k}\Gamma(2s)} \sum_{Q\in Q_D}Q\left(\overline{\tauz},1\right)^k\int_0^\infty e^{-2\pi Dv}\mathcal{M}_{k+\frac{1}{2},s}(4\pi Dv) e^{-4\pi Q_{\tauz}^2 v}v^{k-1}dv\\
=\frac{\left(4\pi D\right)^{\frac{1}{4}-\frac{k}{2}}}{6y^{2k}\Gamma(2s)}\sum_{Q\in \QD}Q\left(\overline{\tauz},1\right)^{k}  I\left(\frac{Dy^2}{\left|Q(\ztau,1)\right|^2}\right),
\end{multline}
where for $0<w<1$ we define
$$
I(w):=\int_0^\infty \mathcal{M}_{k+\frac{1}{2},s}(v)e^{\frac{v}{2}}e^{-vw^{-1}}v^{k-1}dv.
$$
In the case that $s\neq \frac{k}{2}+\frac{1}{4}$, we insert the definition \eqref{eqn:Mdef} of $\mathcal{M}_{k+\frac{1}{2},s}(v)$ and then substitute the integral representation \eqref{eqn:Mintrep} of the $M$-Whittaker function when $\re(s)>\frac{k}{2}$.  The change of variables $t\to 1-t$ yields
\begin{multline*}
I(w)=\frac{\Gamma(2s)}{\Gamma\left(s+\frac{k}{2}+\frac14\right)\Gamma\left(s-\frac{k}{2}-\frac14\right)}\int_0^1(1-t)^{s+\frac{k}{2}-\frac34}t^{s-\frac{k}{2}-\frac{5}{4}}\int_0^\infty v^{s+\frac{k}{2}-\frac{5}{4}}e^{-v\left(w^{-1}-t\right)}dv dt\\
=\frac{\Gamma(2s)\Gamma\left(s+\frac{k}{2}-\frac{1}{4}\right)w^{s+\frac{k}{2}-\frac{1}{4}}}{\Gamma\left(s+\frac{k}{2}+\frac14\right)\Gamma\left(s-\frac{k}{2}-\frac14\right)} \int_0^1(1-t)^{s+\frac{k}{2}-\frac34}t^{s-\frac{k}{2}-\frac54}\left(1-wt\right)^{-\frac{k}{2}-s+\frac14}dt.
\end{multline*}
We then rewrite the integral using the Euler integral representation for ${_2F_1}$ (see (15.3.1) in \cite{AS}), given for $\re(C)>\re(B)>0$ and $|w|<1$ by 
\begin{equation}\label{eqn:Eulerint}
{_2F_1}\left(A,B;C;w\right) = \frac{\Gamma(C)}{\Gamma(B)\Gamma(C-B)}\int_0^1 t^{B-1}(1-t)^{C-B-1}(1-wt)^{-A} dt.
\end{equation}
Thus 
$$
I(w)=\Gamma\left(s+\frac{k}{2}-\frac{1}{4}\right) w^{s+\frac{k}{2}-\frac{1}{4}} {_2F_1}\left(s+\frac{k}{2}-\frac{1}{4},s-\frac{k}{2}-\frac{1}{4};2s;w\right).
$$
Inserting this into \eqref{eqn:unfoldfkDs} shows that $\lim_{\bigT\to\infty}I_1(\bigT)=f_{k,s,D}$.  

For $s=\frac{k}{2}+\frac{1}{4}$, inserting \eqref{eqn:M1F1} into \eqref{eqn:unfoldfkDs} yields
\begin{equation}\label{eqn:Phikhol}
\lim_{\bigT\to\infty} I_1(\bigT)=\frac{D^{\frac{k}{2}+\frac{1}{4}}\Gamma(k)}{6(4\pi)^{\frac{k}{2}-\frac{1}{4}}\Gamma\left(k+\frac{1}{2}\right)}\sum_{Q\in \QD}Q(\tauz,1)^{-k}=f_{k,\frac{k}{2}+\frac{1}{4},D}(\tauz).
\end{equation}

To conclude \eqref{eqn:thetalift1}, it remains to show that $I_2(\bigT)$ vanishes as $\bigT\to\infty$.  We first assume that $4\mid c$ and choose $\gamma=\left(\begin{smallmatrix}a&b\\c&d\end{smallmatrix}\right)\in \Gamma_0\left(4\right)$.  A direct calculation shows that
$$
\gamma S_{\frac{a}{c}} = \left\{\ztau\in \H\Big| v\geq \bigT\right\}.
$$
Hence, the change of variables $\ztau\to \gamma \ztau$, together with the modularity of $\Thpos$ in Proposition \ref{prop:Thetamodular}, yields
$$
I_2(\bigT)=-\frac{1}{6}\int_{\bigT}^{\infty} \int_{-\infty}^{\infty} \overline{\Thpos\left(-\overline{\tauz},\ztau\right)}\left(c\overline{\ztau}+d\right)^{k+\frac{1}{2}}\psi_{D,k+\frac{1}{2}}\left(s;\gamma \ztau\right)\im\left(\gamma \ztau\right)^{k+\frac{1}{2}}\frac{dudv}{v^2}.
$$
Using the facts that $\im\left(\gamma \ztau\right) = \frac{v}{\left|c\ztau+d\right|^2}$ and $\Thpos$ is translation invariant, the integral may be rewritten as
$$
-\frac{1}{6}\int_{\bigT}^{\infty} \int_{0}^{1} \overline{\Thpos\left( -\overline{\tauz},\ztau\right)}\sum_{n=-\infty}^{\infty}\psi_{D,k+\frac{1}{2}}\left(s;\ztau\right)\Big|_{k+\frac{1}{2}}\gamma \left(\begin{matrix}1&n\\0&1\end{matrix}\right) v^{k+\frac{1}{2}}\frac{dudv}{v^2}.
$$
Taking the sum over all $a,c$ with $4\mid c> 0$, the inner sum precisely evaluates as
$$
\PP_{k+\frac{1}{2},s,\Gamma_0(4),D}-\psi_{D,k+\frac{1}{2}}\left(s;\ztau\right).
$$
Comparing the polynomial growth in \eqref{eqn:Poincgrowth} with the exponential decay of $\Thpos\left( -\overline{\tauz},\ztau\right)$ towards $i\infty$, one concludes that the limit $\bigT\to\infty$ vanishes.  A similar argument shows that the contribution to $I_2(\bigT)$ coming from $4\nmid c$ also vanishes as $\bigT\to\infty$.  This yields \eqref{eqn:thetalift1}.

\end{proof}

\section{Image of the theta lift $\Phineg$}\label{sec:Phi1-kimage}
We next compute the image of $\Phineg$ with the method from Section \ref{sec:Phikimage}.
\begin{proof}[Proof of Theorem \ref{thm:negwt} (2)]
Following the argument in the proof of Theorem \ref{thm:poswt} (2), we may reduce the theorem to evaluating 
\begin{multline}\label{Zagiersplit2}
\frac{1}{6}\lim_{\bigT\to\infty}\int_0^{\bigT} \int_0^1 \psi_{D,\frac{3}{2}-k}\left(s;\ztau\right) \overline{\Thneg\left(-\overline{\tauz},\ztau\right)}v^{\frac32-k}\frac{dudv}{v^2}\\
-\frac{1}{6}\lim_{\bigT\to\infty}\sum_{c\geq 1}\sum_{\substack{a\; (\text{mod} \; c)\\ (a,c)=1}}\int_{S_{\frac{a}{c}}(\bigT)}\psi_{D,\frac{3}{2}-k}\left(s;\ztau\right)\overline{\Thneg\left(-\overline{\tauz},\ztau\right)} v^{\frac32-k}\frac{dudv}{v^2}.
\end{multline}
Using the argument from before, the second summand vanishes.  We use \eqref{eqn:Qrewrite} to rewrite the exponential in the theta series as
$$
\left(b^2-4ac\right)u+iv\left(2Q_{-\overline{\tauz}}^2+\left(b^2-4ac\right)\right).
$$
Therefore, evaluating the integral over $u$ and then making the change of variables $Q\to \widetilde{Q}$ (as defined before \eqref{eqn:Qtilderef}), it suffices to compute
\begin{equation}\label{sintegral}
\frac{1}{6}(4\pi D)^{\frac{1}{4}-\frac{k}{2}}\Gamma(2s)^{-1}\sum_{Q\in \QD} Q_{\tauz} Q(\tauz,1)^{k-1}\mathcal{I}\left(\frac{Dy^2}{|Q(\tauz,1)|^2}\right),
\end{equation}
where
$$
\mathcal{I}(w):=\int_0^\infty\mathcal{M}_{\frac{3}{2}-k,s}(-v)\, e^{\frac{v}{2}}\,v^{-\frac12}\,e^{-vw^{-1}}dv.
$$

Inserting the definition \eqref{eqn:Mdef} of $\mathcal{M}_{\frac{3}{2}-k,s}$ and the integral representation \eqref{eqn:Mintrep} of the $M$-Whittaker function, we evaluate
\begin{multline}\label{eqn:neglift}
\mathcal{I}(w)=\frac{\Gamma(2s)}{\Gamma\left(s-\frac{k}{2}+\frac{3}{4}\right)\Gamma\left(s+\frac{k}{2}-\frac{3}{4}\right)}\int_0^1 t^{s+\frac{k}{2}-\frac74}(1-t)^{s-\frac{k}{2}-\frac14}\,\int_0^\infty v^{s+\frac{k}{2}-\frac{5}{4}}e^{-v\big(t-1+w^{-1}\big)}dvdt\\
 =\frac{\Gamma(2s)\Gamma\left(s+\frac{k}{2}-\frac14\right)}{\Gamma\left(s-\frac{k}{2}+\frac{3}{4}\right)\Gamma\left(s+\frac{k}{2}-\frac{3}{4}\right)}w^{s+\frac{k}{2}-\frac{1}{4}} \int_0^1(1-t)^{s+\frac{k}{2}-\frac74} t^{s-\frac{k}{2}-\frac14}\left(1-wt\right)^{-s-\frac{k}{2}+\frac{1}{4}}dt. 
\end{multline}
We again employ the Euler integral representation \eqref{eqn:Eulerint} to show that 
$$
\mathcal{I}(w)=\Gamma\left(s+\frac{k}{2}-\frac14\right)w^{s+\frac{k}{2}-\frac{1}{4}}{_2F_1}\left(s+\frac{k}{2}-\frac{1}{4},s-\frac{k}{2}+\frac{3}{4};2s;w\right).
$$
We then rewrite the hypergeometric function by using the Euler transform
$$
{_2F_1}\left(A,B;C;w\right)=\left(1-w\right)^{C-A-B}{_2F_1}\left(C-A,C-B;C;w\right)
$$
to yield
$$
\mathcal{I}(w)=\Gamma\left(s+\frac{k}{2}-\frac14\right)(1-w)^{-\frac{1}{2}}w^{s+\frac{k}{2}-\frac{1}{4}}{_2F_1}\left(s-\frac{k}{2}+\frac{1}{4},s+\frac{k}{2}-\frac{3}{4};2s;w\right).
$$
Finally, we conclude that \eqref{sintegral} equals \eqref{eqn:thetalift1} by using \eqref{eqn:Qrewrite} to rewrite $\left|Q_{\tauz}\right|$ in terms of $\frac{Dy^2}{|Q(\tauz,1)|^2}$. 

\end{proof}

We next establish the injectivity of the lift.
\begin{proof}[Proof of Theorem \ref{thm:Heckeneg} (2)]
Since the Poincar\'e series $\PP_{\frac{3}{2}-k,s,D}$ span the space of weak Maass forms and are linearly independent (which can be seen by comparing their principal parts), it is enough to show that the functions $\mathcal{F}_{1-k,s,D}$ are linearly independent.
This follows by proving that any linear combination
$$
\mathcal{F}:=\sum_{j=1}^{n} a_j \mathcal{F}_{1-k,s,D_j}
$$
with $a_j$ not all zero exhibits discontinuities and is hence nonzero.  Comparing the sets $E_{D_j}$ of geodesics defined in \eqref{eqn:EDdef} implies the result.
\end{proof}

\section{Relation between positive and negative weight local Maass forms}\label{sec:negposwt}

In this section we relate $f_{k,s,D}$ and $\mathcal{F}_{1-k,s,D}$.
\begin{proof}[Proof of Theorem \ref{thm:negposrel}]
We prove \eqref{eqn:xiFs} by establishing that for $P:=\PP_{\frac{3}{2}-k,s,D}$ and $z\notin E_D$, one has 
\begin{equation}\label{eqn:xiFs2}
\xi_{2-2k}\left(\Phineg\left(P\right)(\tauz)\right)= 2\Phipos\left(\xi_{\frac{3}{2}-k}(P)\right)(\tauz).
\end{equation}
We first use \eqref{eqn:xicommuteTh*} and then \eqref{eqn:thetaid1} to obtain for $z\notin E_D$
\begin{multline}\label{eqn:Th*zThtau}
\Phipos\left(\xi_{\frac{3}{2}-k}(P)\right)(\tauz) = \left< \xi_{\frac{3}{2}-k}\left(P\right), \Thpos\left(\tauz,\cdot\right)\right>^{\text{reg}} =-\overline{\left<P, \xi_{k+\frac{1}{2}}\left(\Thpos\left(\tauz,\cdot\right)\right)\right>^{\text{reg}}}\\
=-\overline{\left< P,-i y^{2-2k}\frac{\partial}{\partial\tauz}\Thneg\left(-\overline{\tauz},\cdot\right)\right>^{\text{reg}}}=\frac{iy^{2-2k}}{6}\frac{\partial}{\partial\tauz}\int_{\mathcal{F}_0(4)}^{\text{reg}} \overline{P(\ztau)} \Thneg\left(-\overline{\tauz},\ztau\right)v^{\frac{3}{2}-k} \frac{dudv}{v^2}\\
 =iy^{2-2k}\frac{\partial}{\partial\tauz}\overline{\left< P,\Thneg\left(-\overline{\tauz},\cdot\right)\right>^{\text{reg}}}.
\end{multline}
Since
\begin{equation}\label{eqn:xidef2}
\xi_{\wt}\left(G(\tauz)\right) = 2iy^{\wt}\frac{\partial}{\partial{\tauz}}\overline{G(\tauz)},
\end{equation}
we conclude \eqref{eqn:xiFs2} from \eqref{eqn:Th*zThtau}. 

We now apply Theorem \ref{thm:negwt} (2), \eqref{eqn:xiFs2}, \eqref{eqn:xiPs}, and finally Theorem \ref{thm:poswt} (2) to yield
\begin{multline*}
\xi_{2-2k}\left(\mathcal{F}_{1-k,s,D}(\tauz)\right)= \xi_{2-2k}\left(\Phineg\left(\PP_{\frac{3}{2}-k,s,D}\right)(\tauz)\right)= 2\Phipos\left(\xi_{\frac{3}{2}-k}\left(\PP_{\frac{3}{2}-k,s,D}\right)\right)(\tauz)\\
=2\left(\overline{s}-\frac{3}{4}+\frac{k}{2}\right) \Phipos\left(\PP_{k+\frac{1}{2},\overline{s},D}\right)(\tauz)= 2\left(\overline{s}-\frac{3}{4}+\frac{k}{2}\right) f_{k,\overline{s},D}(\tauz).
\end{multline*}
This concludes the proof of \eqref{eqn:xiFs}.

We next prove \eqref{eqn:xi2k}.  Denoting $P:=\PP_{k+\frac{1}{2},s,D}$, we use \eqref{eqn:xicommuteTh} to conclude that for $z\notin E_D$
\begin{equation}\label{eqn:xiswitch2}
\Phineg\left(\xi_{k+\frac{1}{2}}(P)\right)(\tauz)=\left< \xi_{k+\frac{1}{2}}\left(P\right), \Thneg\left(-\overline{\tauz},\cdot\right)\right>^{\text{reg}} = -\overline{\left< P, \xi_{\frac{3}{2}-k}\left(\Thneg\left(-\overline{\tauz},\cdot\right)\right)\right>^{\text{reg}}}.
\end{equation}
We then employ \eqref{eqn:xiThtauTh*z} and \eqref{eqn:xidef2} to obtain
\begin{equation}\label{eqn:Phi1-kPhik}
\Phineg\left(\xi_{k+\frac{1}{2}}(P)\right)(\tauz)=iy^{2k}\frac{\partial}{\partial \tauz}\overline{ \left< P,\Thpos\left(\tauz,\cdot\right)\right>^{\text{reg}}}= \frac{1}{2}\xi_{2k}\left(\Phipos(P)(\tauz)\right).
\end{equation}
Combining this with Theorem \ref{thm:negwt} (2), \eqref{eqn:xiPs}, and Theorem \ref{thm:poswt} (2) yields
\begin{multline*}
\left(\overline{s}-\frac{k}{2}-\frac{1}{4}\right)\mathcal{F}_{1-k,\overline{s},D}(\tauz)=  \left(\overline{s}-\frac{k}{2}-\frac{1}{4}\right)\Phineg\left(P_{\frac{3}{2}-k,\overline{s},D}\right)(\tauz)\\
=\Phineg\left(\xi_{k+\frac{1}{2}}\left(\PP_{k+\frac{1}{2},s,D}\right)\right)(\tauz)=\frac{1}{2}\xi_{2k}\left(\Phipos\left(\PP_{k+\frac{1}{2},s,D}\right)(\tauz)\right)= \frac{1}{2}\xi_{2k}\left(f_{k,s,D}(\tauz)\right).
\end{multline*}
\end{proof}

We are now ready to prove Theorem \ref{thm:poswt} (1) and Theorem \ref{thm:negwt} (1).
\begin{proof}[Proof of Theorem \ref{thm:poswt} (1)]
Note that 
$$
\overline{\Thpos\left(\tauz,\ztau\right)}=\Thpos\left(-\overline{\tauz},-\overline{\ztau}\right).
$$
Hence $f_{k,s,D}$ is modular of weight $2k$ by Proposition \ref{prop:Thetamodular}.  

The functions $f_{k,s,D}$ are continuous since for  $\re(C)>\re(A+B)$, the hypergeometric function ${_2F_1}\left(A,B;C;w\right)$, and hence $\phipos_s(w)$, is continuous for $w\leq 1$.  This implies condition (3).

For $\tauz\notin E_D$, \eqref{eqn:xi2k} and \eqref{eqn:xiFs} imply that
$$
\Delta_{2k}\left(f_{k,s,D}(\tauz)\right) = -\xi_{2-2k}\left(\xi_{2k}\left(f_{k,s,D}(\tauz)\right)\right) = 4\lambda_s f_{k,s,D}(\tauz).
$$
A straightforward calculation shows that $f_{k,s,D}(\tauz)$ grows at most polynomially as $y\to \infty$.

Finally, one uses \eqref{eqn:Phikhol} and the duplication formula for the $\Gamma$-function to conclude \eqref{eqn:fkDcusp}.

\end{proof}
\begin{remark}
The non-differentiability of $f_{k,s,D}$ follows by using \eqref{eqn:xi2k} and then proving that the functions $\mathcal{F}_{1-k,s,D}$ are not continuous.  Computational evidence indicates that $f_{k,s,D}(\tauz)$ decays exponentially as $y\to\infty$.
\end{remark}

\begin{proof}[Proof of Theorem \ref{thm:negwt} (1)]
Noting that 
$$
\overline{\Thneg\left(-\overline{\tauz},\ztau\right)}=\Thneg\left(\tauz,-\overline{\ztau}\right),
$$
Proposition \ref{prop:Thetamodular} implies that $\mathcal{F}_{1-k,s,D}$ is modular of weight $2-2k$.  

The proof that $\mathcal{F}_{1-k,s,D}$ is an eigenfunction under $\Delta_{2-2k}$ with eigenvalue $4\lambda_s$ follows by \eqref{eqn:xiFs} and \eqref{eqn:xi2k} precisely as in the proof of Theorem \ref{thm:poswt} (1).

In order to show condition (3) in the definition of local Maass forms, we first note that $\phineg_s(w)$ is continuous for $0<w\leq 1$.  The locally uniform convergence of the sum allows us to pull the limit $r\to 0^+$ of $\mathcal{F}_{1-k,s,D}(\tauz\pm ir)$ into each term.  Define 
$$
\mathscr{B}_{\tauz}:=\left\{ Q\in \QD\big| Q_{\tauz}=0\right\}.
$$
By Lemma 5.1 of \cite{BKW}, there are only finitely many $Q\in \mathscr{B}_{\tauz}$.  Note that
$$
\sgn\left(Q_{\tauz}\right)=\sgn\left(Q_{\tauz\pm ir}\right)
$$ 
for $r$ sufficiently small and $Q\notin \mathscr{B}_{\tauz}$, while for $Q\in \mathscr{B}_{\tauz}$ one has
$$
\sgn\left(Q_{\tauz+ ir}\right)=-\sgn\left(Q_{\tauz- ir}\right).
$$
Hence, since the terms of $\mathcal{F}_{1-k,s,D}(\tauz)$ with $Q\in \mathscr{B}_{\tauz}$ vanish,
\begin{multline*}
\frac{1}{2}\lim_{r\to 0^+} \left(\mathcal{F}_{1-k,s,D}(\tauz+ ir)+\mathcal{F}_{1-k,s,D}(\tauz-ir)\right) = \sum_{Q\notin \mathscr{B}_{\tauz}} \sgn\left(Q_{\tauz}\right)Q(\tauz,1)^{k-1}\phineg_s\left(\frac{Dy^2}{|Q(\tauz,1)|^2}\right)\\
=\mathcal{F}_{1-k,s,D}(\tauz).
\end{multline*}

A direct calculation shows that $\mathcal{F}_{1-k,s,D}(\tauz)$ grows at most polynomially as $y\to\infty$.
\end{proof}
\begin{remark}
To show that $\mathcal{F}_{1-k,s,D}$ exhibits discontinuities along the set $E_D$, one computes 
$$
\lim_{r\to 0^+} \left(\mathcal{F}_{1-k,s,D}(\tauz+ ir)-\mathcal{F}_{1-k,s,D}(\tauz-ir)\right)
$$
 similarly as in the proof of Theorem \ref{thm:negwt} (1).  It is shown to be nonzero by using Gauss's summation formula to conclude that $\phineg_s(1)\neq 0$.  

If $D$ is not a square and $\re(s)\geq \frac{k}{2}+\frac{1}{4}$, then computational evidence indicates that $\mathcal{F}_{1-k,s,D}$ is bounded as $y\to\infty$.
\end{remark}

\section{Hecke operators}\label{sec:Hecke}
In this section, we consider the action of the Hecke operators on the theta lifts.
\begin{proof}[Proof of Theorem \ref{thm:Heckeneg} (1)]
Since the Poincar\'e series span the space of weight $\frac{3}{2}-k$ weak Maass forms, it suffices to compute the action of the Hecke operators on Poincar\'e series.  As in the proof of Theorem 1.5 of \cite{BKW}, one can show that
$$
\mathcal{F}_{1-k, s,D}\Big|_{2-2k} T_p=\mathcal{F}_{1-k, s,Dp^2}+p^{-k}\left(\frac{D}{p}\right)
\mathcal{F}_{1-k, s,D}+p^{1-2k}\mathcal{F}_{1-k, s,\frac{D}{p^2}}.
$$
Hence by Theorem \ref{thm:negwt} (1), equation \eqref{eqn:Hecke} follows by the easily verified identity
\[
\PP_{\frac32-k, s,D}\Big|_{\frac{3}{2}-k} T_{p^2}=\PP_{\frac32-k, s,Dp^2}+p^{-k}\left(\frac{D}{p}\right)\PP_{\frac32-k, s,D}+p^{1-2k} \PP_{\frac32-k, s,\frac{D}{p^2}}.
\]
\end{proof}
We now move on to the positive weight case.
\begin{proof}[Proof of Theorem \ref{thm:Heckepos}]
We first prove Theorem \ref{thm:Heckepos} (1).  Let $\mathcal{H}$ be a weight $2k$ local Maass form with exceptional set $E_D$ which is continuous everywhere.  Since continuity is preserved by the Hecke operators, one easily checks that $\mathcal{H}|_{2k}T_p$ is a local Maass form.  To determine the exceptional set for $\mathcal{H}$, recall that the weight $2k$ Hecke operator may be written as
$$
\mathcal{H}\Big|_{2k}T_p\left(\tau\right) = p^{2k-1} \mathcal{H}\left(p\tau\right)+ p^{-1}\sum_{r\pmod{p}} \mathcal{H}\left(\frac{\tau+r}{p}\right).
$$
By computing the image of $E_D$ under $\tau \to p\tau$ and $\tau \to \frac{\tau+r}{p}$, one concludes that $\mathcal{H}|_{2k}T_p$ has exceptional set $E:=E_{Dp^2}\supset E_D$.  Hence it suffices to prove the statement for $\tauz\notin E$.

Suppose that $H$ is a weak Maass form of weight $k+\frac{1}{2}$ with eigenvalue $\lambda_s$ under $\Delta_{k+\frac{1}{2}}$.  Since $\xi_{\frac{3}{2}-k}$ surjects onto the space of weak Maass forms of weight $k+\frac{1}{2}$ with eigenvalue $\lambda_{s}$ (see \cite{BF}), we may choose such a weight $\frac{3}{2}-k$ weak Maass form $G$ such that $\xi_{\frac{3}{2}-k}(G)=H$.  But then by \eqref{eqn:Hecke}, \eqref{eqn:xiFs2}, and the fact that the Hecke operators commute with $\xi_{2-2k}$, for $\tauz\notin E$, we have that
$$
\Phipos\left(H\right)\Big|_{2k}T_{p}(\tauz)=\frac{1}{2} \xi_{2-2k}\left(\Phineg\left(G\right)\right)\Big|_{2k}T_{p}(\tauz)  =\frac{1}{2} \xi_{2-2k}\left(\Phineg\left(G\Big|_{\frac{3}{2}-k}T_{p^2}\right)(\tauz)\right).
$$
We now use \eqref{eqn:Phi1-kPhik} and the fact that $\xi_{\frac{3}{2}-k}$ commutes with the Hecke operators to obtain
$$
\xi_{2-2k}\left(\Phineg\left(G\Big|_{\frac{3}{2}-k}T_{p^2}\right)(\tauz)\right)=2\Phipos\left(\xi_{\frac{3}{2}-k}\left(G\Big|_{\frac{3}{2}-k}T_{p^2}\right)\right)(\tauz)=2\Phipos\left(H\Big|_{k+\frac{1}{2}}T_{p^2}\right)(\tauz),
$$
as desired.

We move on to Theorem \ref{thm:Heckepos} (2).  Assume that $\Phipos(F)\equiv 0$ for a weak Maass form $F$ with eigenvalue $\lambda_s\neq 0$.  Writing $G:=-(4\lambda_s)^{-1} \xi_{k+\frac{1}{2}}(F)$, by \eqref{eqn:xiFs2} we have that
$$
0=\xi_{2-2k}\left(\Phineg(G)\right).
$$
Since $\Phineg(G)$ is an eigenfunction under $\Delta_{2-2k}$ with eigenvalue $4\lambda_s\neq 0$, we have
$$
0= -\left(4\lambda_s\right)^{-1}\xi_{2k}\left(\xi_{2-2k}\left(\Phineg(G)\right)\right)= \Phineg(G).
$$
Since $\Phineg$ is injective, we conclude that $G\equiv 0$.  However,
$$
\xi_{\frac{3}{2}-k}(G)=F,
$$
and hence $F\equiv 0$.
\end{proof}

\end{document}